\documentclass[12 pt]{amsart}
\usepackage{amsmath,amssymb,amsthm,amscd,enumerate,setspace}
\usepackage{hyperref}
\hypersetup{colorlinks=true, linkcolor=red, citecolor=blue} 
\usepackage{graphicx}
\usepackage{xcolor}

\long\def\symbolfootnote[#1]#2{\begingroup%
\def\thefootnote{\fnsymbol{footnote}}\footnote[#1]{#2}\endgroup}
\input xypic

\newtheorem{Theorem}{Theorem}[section]
\newtheorem{Lemma}[Theorem]{Lemma}
\newtheorem{Corollary}[Theorem]{Corollary}
\newtheorem{Proposition}[Theorem]{Proposition}

\newtheorem{Remark}[Theorem]{Remark}
\newtheorem{Example}[Theorem]{Example}
\newtheorem{Definition}[Theorem]{Definition}

\def\F{\mathcal{F}}
\def\FF{\mbox{\rm F}}
\def\G{{\mathrm G}}

\def\m{\mathfrak{m}}

\def\PP{\mbox{\rm P}}
\def\p{\mathfrak{p}}

\def\q{\mathfrak{q}}
\def\r{\mathrm{r}}
\def\R{{\mathcal R}}

\def\S{{\mathrm S}}

\def\AN{\mbox{\rm AN}}
\def\ann{\mbox{\rm ann}}

\def\core{\mbox{\rm core}}
\def\depth{\mbox{\rm depth}}
\def\ds{\displaystyle}

\def\Ext{\mbox{\rm Ext}}

\def\grade{\mbox{\rm grade}}
\def\h{\mbox{\rm ht}}
\def\Hom{\mbox{\rm Hom}}

\def\lar{\longrightarrow}
\def\Min{\mbox{\rm Min}}

\def\pdim{\mbox{\rm pdim}}

\def\rank{\mbox{\rm rank}}
\def\rar{\rightarrow}

\def\Spec{\mbox{\rm Spec}}
\def\Supp{\mbox{\rm Supp}}

\def\Sym{\mathcal{S}}

\newcommand{\rees}[1]{\ensuremath{\mathcal{R}(#1)}}

\begin{document}

\title{Residual Intersections and Core of Modules}

\author{A. Costantini}
\address{Department of Mathematics, Oklahoma State University, 401 MSCS, Stillwater, OK 74078, USA }
\email{alecost@okstate.edu }

\author{L. Fouli}
\address{Department of Mathematics, New Mexico State University, PO Box 30001, Las Cruces, NM 88003-8001, USA }
\email{lfouli@nmsu.edu }

\author{J. Hong}
\address{Department of Mathematics, Southern Connecticut State
University, 501 Crescent Street, New Haven, CT 06515-1533, U.S.A.} 
\email{hongj2@southernct.edu}

\thanks{AMS 2020 {\em Mathematics Subject Classification}.
Primary 13H10;  Secondary 13A30, 13C13, 13D07 
The third author was partially supported by the Sabbatical Leave Program at Southern Connecticut State University (Spring 2022).  }

\begin{abstract}
We introduce the notion of residual intersections of modules and prove their existence. 
We show that projective dimension one modules have Cohen-Macaulay residual intersections, namely they satisfy the relevant Artin-Nagata property. We then establish a formula for the core of orientable modules satisfying certain homological conditions, extending previous results of Corso, Polini, and Ulrich on the core of projective one modules. Finally, we provide examples of classes of modules that satisfy our assumptions.
\end{abstract}

\maketitle

\section{Introduction}

The main goal of this article is to describe the {\em core} of a module, namely the (possibly {\em infinite}) intersection of all minimal reductions of the module. This notion generalizes the core of an ideal, which is defined as the intersection of all minimal reductions of the ideal. 
Minimal reductions of an ideal $I$ and its core have been studied by several authors using various techniques \cite{{CPU01},{CPU02},{FM12},{HSwanson95},{Lipman94},{RS88}}.  In particular, Corso, Polini, and Ulrich \cite{{CPU01},{CPU02}} established a fundamental relationship between the structure of the core of   $I$ and the properties of the residual intersections of $I$. Generalizing the notion of linkage, the residual intersections of $I$ are the ideals of the form $(J :_{R} I)$ for some ideal $J \subseteq I$ that equals $I$ locally up to a certain codimension. Hence, one can think of the residual intersections of $I$ as local approximations of $I$.  

\medskip

The structure of the core is well-understood for ideals whose residual intersections satisfy Serre's condition $(\S_{2})$. In fact, Corso, Polini, and Ulrich proved that the core of $I$ is a {\em finite} intersection of general minimal reductions of $I$ if $I$ is $\G_{\ell(I)}$ and universally weakly $(\ell(I)-1)$-residually $(\S_{2})$, where $\ell(I)$ is the analytic spread of $I$, \cite[Theorem~4.5]{CPU01}. They also gave an explicit formula to calculate the core for ideals with bounded reduction number in a slightly more restrictive setting \cite{CPU02}.

\begin{Theorem} \cite[Theorem~2.6]{CPU02} \label{TheoremCPU02}
   Let $R$ be a Gorenstein local ring of dimension $d$ with infinite residue field. Let $I$ be an ideal with height $g \ge 2$ and analytic spread $\ell$, satisfying $\G_{\ell}$. 
Assume that $\depth(I^j) \ge d - g - j +2$ for $1 \le j \le \ell - g +1$.
   The following are equivalent{\rm \,:}
   \begin{enumerate}[{\rm (i)}]
     \item $\core(I) = (J :_R I)I = (J :_R I )J$ for every minimal reduction $J$ of $I${\rm ;} 
      \item $(J :_R I)$ does not depend on the minimal reduction $J$ of $I${\rm ;}
      \item The reduction number of $I$ is at most  $ \ell - g +1$.
   \end{enumerate}
\end{Theorem}

The condition that $(J:_RI)$ does not depend on the choice of minimal reduction $J$ of $I$ is also known as the {\em balanced} condition, \cite{U96}. The assumption on the depth of the  powers of $I$ in Theorem~\ref{TheoremCPU02} is satisfied if $I$ is a strongly Cohen-Macaulay ideal \cite[Remark~2.10]{U94}, or an ideal in the linkage class of a complete intersection such as a perfect ideal of height two \cite{Huneke82}. This assumption implies that $I$ satisfies the {\em Artin-Nagata property} $\AN_{\ell(I)}$, a Cohen-Macaulayness condition on the residual intersections of $I$ \cite[Theorem~2.9]{U94}. Furthermore, if $I$ is a perfect ideal of height two, the equivalent conditions in Theorem~\ref{TheoremCPU02} are also equivalent to the Cohen-Macaulay property of the Rees algebra $\R(I)$ \cite[Corollary~5.3]{U96}, \cite[Corollary~3.4]{CPU02}.  Exploiting this Cohen-Macaulay property, Corso, Polini and Ulrich extended the latter result to modules of projective dimension one in \cite{CPU03}.

\begin{Theorem}\cite[Theorem~2.4]{CPU03} \label{TheoremCPU03}  
Let $R$ be a Gorenstein local ring of dimension $d$ with infinite residue field. Let $E$ be a finitely generated $R$-module with ${\rm{projdim }}E=1$. Let $e=\rank (E)$, $\ell=\ell(E)$, and assume that $E$ satisfies $G_{\ell-e+1}$ and is torsionfree locally in codimension $1$. The following are equivalent{\rm \,:}
 \begin{enumerate}[{\rm (a)}]
\item $(U:_RE)E \subseteq  \core (E)$ for some minimal reduction $U$ of $E$;
\item $(U:_RE)U=(U:_RE)E=\core (E)$ for every minimal reduction $U$ of $E$;
\item $\core(E)={\rm{Fitt}}_{\ell}(E)\cdot E$;
\item $(U:_RE)$ does not depend on the minimal reduction $U$ of $E$;
\item $(U:_RE)={\rm{Fitt}}_{\ell}(E)$ for every minimal reduction $U$ of $E$;
\item the reduction number of $E$ is at most $\ell-e$. 
\end{enumerate}
\end{Theorem}

If we assume that the module $E$ in the above result has rank one, then $E$ is isomorphic to a  perfect ideal of height two, recovering a special case of Theorem~\ref{TheoremCPU02}. Nonetheless, Theorem~\ref{TheoremCPU02} holds for larger classes of ideals and it is then natural to ask whether Theorem~\ref{TheoremCPU03} can be extended to other classes of modules. 

\medskip

Let us recall some basic definitions and notations used throughout this article. Let $R$ be a Noetherian ring and $E$ a finitely generated $R$--module. We say that $E$ has a {\em rank}, $\mathrm{rank}(E) = e$, if $E \otimes_R \mathrm{Quot(R)}$ is a finite free module of rank $e$ over the total ring of quotients $\mathrm{Quot(R)}$. This condition is not very restrictive. For instance, all finitely generated modules over an integral domain and all modules of finite projective dimension over any Noetherian ring have a rank. Moreover, ideals of positive grade are identified with torsionfree modules of rank one.   For an integer $s \ge 1$, a finitely generated $R$--module $E$ having a rank $e$ is said to satisfy $\G_{s}$ if ${\ds \mu(E_{\p}) \le \dim(R_{\p}) + e-1}$ whenever ${\ds 1 \le \dim(R_{\p}) \le s-1}$, where $\mu(-)$ denotes the minimal number of generators. One says that $E$ satisfies $\G_{\infty}$ if it satisfies $\G_{s}$ for all $s$. 

\medskip

 If $E$ is a finitely generated module with a rank, then the {\em Rees algebra}  $\R(E)$ of $E$ is defined as the symmetric algebra $\Sym(E)$ modulo its $R$--torsion submodule (see \cite{EHU03} for more general definition of Rees algebras of modules that holds for modules without a rank as well). If $\Sym(E) = \R(E)$, then $E$ is said to be of {\em linear type} since the symmetric algebra can be equivalently described as a quotient of a polynomial ring modulo an ideal generated by homogeneous linear polynomials. 
 
 \medskip

 A submodule $U$ of $E$ is called a {\em reduction} of $E$ if $\R(E)$ is integral over the $R$--subalgebra generated by the image of $U$.  A reduction of $E$ is said to be {\em minimal} if it is minimal with respect to inclusion.  The {\em core} of a module $E$ is the intersection of all (minimal) reductions of $E$ and denoted by $\core(E)$. Let $U$ be a reduction of $E$.  The {\em reduction number} $\r_{U}(E)$ of $E$ relative to $U$ is the least integer $r \ge 0$ such that ${\ds \R(E)_{n+1}= U \cdot \R(E)_{n} }$ for all $n \ge r$. The {\em reduction number} $\r(E)$ is the minimum of $\r_{U}(E)$, where $U$ ranges over all minimal reductions of $E$.
 
\medskip

Let $(R, \m)$ be a Noetherian local ring with residue field $k$. The Krull dimension of the {\em special fiber ring} ${\ds \F(E) = k \otimes_{R} \R(E)}$ is called the {\em analytic spread} of $E$ and is denoted by $\ell(E)$.
If $k$ is infinite,  $\ell(E) =\mu(U)$ for any minimal reduction $U$ of $E$. Moreover, 
$\ell(E) \le \mu(E)$, and equality holds if and only if $E$ has no proper reductions. 

\medskip

We now describe how this article is organized. In Section~\ref{SectionResInt}, we introduce the notion of {\em residual intersections of modules} and prove their existence under the $G_s$ assumption, generalizing the case of ideals satisfying $G_s$ \cite{U94}.
In Section~\ref{SectionAN}, we show that a module $E$ of projective dimension $1$ satisfies the Artin-Nagata property $\AN_{s}$ for every $s$. This would explain the formula for $\core(E)$ given by Theorem~\ref{TheoremCPU03} in terms of residual intersections of $E$, similarly as for perfect ideals of height two. In Section~\ref{CoreM}, we extend Theorem~\ref{TheoremCPU02} and  Theorem~\ref{TheoremCPU03} to orientable modules satisfying certain homological conditions.  More specifically, our main result is the following. 

\medskip

\noindent {\bf Theorem~\ref{core}.}{\it \; Let $R$ be a Gorenstein local ring of dimension $d$ with infinite residue field. Let $E$ be a finite, torsionfree, and orientable $R$-module with $\rank(E)=e>0$ and analytic spread $\ell(E)=\ell \ge e+2$. Assume that $E$ is $\G_{\ell-e+1}$ and that ${\ds \Ext_{R}^{\,j+1}(E^j, R)= 0}$ for every ${\ds 1 \le j \le \ell-e-1}$. Suppose that the Rees algebra $\R(E)$ satisfies $(\S_{2})$.
The following are equivalent{\rm \,:}
\begin{enumerate}[{\rm (i)}]
\item $(U:_RE)E \subseteq \core(E)$ for every minimal reduction $U$ of $E${\rm \,;}
\item $(U:_RE)U=(U:_RE)E=\core(E)$ for every minimal reduction  $U$ of $E${\rm \,;}
\item$(U:_RE)$ is independent of $U$ for every minimal reduction $U$ of $E$.
\end{enumerate}}

\noindent As applications of our main result, we list several classes of modules to which Theorem~\ref{core} can be applied; see Theorem~\ref{balancedI} and Corollaries~\ref{scmh2},~\ref{h2depthcase},~\ref{ModCorWithdepth}.

\medskip

\section{Residual Intersections of Modules}\label{SectionResInt}

In this section, we introduce residual intersections of modules and discuss their existence. This work generalizes previous work of Ulrich, \cite{U94} on the existence of residual intersections for ideals. 

\begin{Definition}\label{resmod}{\rm
Let $R$ be a Noetherian ring, $E$ a finitely generated $R$--module having a rank $e>0$, and $s$ an integer with $s \ge e$.
 A proper $R$--ideal $K$ is called an {\em $s$--residual intersection} of $E$ if there exists an $s$--generated  submodule $U \subseteq E$ such that ${\ds K=(U:_{R} E)}$ and ${\ds \h(K) \ge s-e+1}$, where $\h(-)$ denotes the height of an ideal. If $E$ is an $R$--ideal $I$ of positive grade, then we assume that $s \ge \h(I)$.
}\end{Definition}

Similarly as in the case of ideals \cite[Corollary~1.6]{U94}, the $\G_{s}$ condition  plays a crucial role in proving that residual intersections of modules exist.

\medskip

\begin{Theorem}\label{exresmod}
Let $(R, \m)$ be a Noetherian local ring with infinite residue class field. 
Let $\mathcal{P}$ be a finite {\rm(}possibly empty{\rm)} subset of $\Spec(R)$.
Let  $E$ be a finitely generated  $R$--module having a rank $e>0$.
Let $s$ be an integer.
Let $W$ be a submodule of $E$ such that  ${\ds \h( W :_{R} E) \ge s}$. Suppose that $E$ satisfies $\G_{s}$.Then there exist elements ${\ds a_{1}, \ldots, a_{s}}$ in $W$ such that for every $0 \le i \le s$, and every subset ${\ds \{ v_{1}, \ldots, v_{i} \} }$ of ${\ds  \{1, \ldots, s \}}$,
 \begin{enumerate}[{\rm (1)}]
 \item ${\ds \mu \left( \left( W/ \sum_{j=1}^{s} Ra_{j} \right)_{\p} \right) = \max\{0, \, \mu(W_{\p}) -s \} }$ whenever ${\ds \p \in \mathcal{P}}${\rm ;}
 \item ${\ds \h \left(  \sum_{j=1}^{i} Ra_{v_{j}} :_{R} E \right)  \ge i-e+1 }$.
 \end{enumerate}
\end{Theorem}

\begin{proof} We proceed by induction on $k$, where $0 \le k \le s$, to construct elements ${\ds a_{1}, \ldots, a_{k} \in W}$ such that, 
for all $0 \le i \le k$ and every subset ${\ds \{v_{1}, \ldots, v_{i}\}  \subseteq \{1, \ldots, k\} }$, we have
\begin{enumerate}[(a)]
\item ${\ds \mu \left( \left( W/ \sum_{j=1}^{k} Ra_{j} \right)_{\p} \right) = \max\{0, \, \mu(W_{\p} )-k \} }$ whenever ${\ds \p \in \mathcal{P}}${\rm ;}
\item ${\ds \mu \left(  \left( W/\sum_{j=1}^{i} Ra_{v_{j}}  \right)_{\p} \right) \le  \dim(R_{\p}) +e-1-i}$ for all $\p \in \Supp(E)$ with ${i-e+1 \le \dim(R_{\p}) \le s-1}${\rm ;}
\item ${\ds  \left( W/\sum_{j=1}^{i} Ra_{v_{j}}  \right)_{\p} =0 }$ for all prime ideals $\p$ with $\h(\p) \le i-e$.
\end{enumerate}

\noindent Suppose $k=0$. It is enough to prove (b).  For all ${\ds \p}$ with ${\ds \dim(R_{p}) \le s-1}$ we have ${\ds E_{\p}=W_{p}}$. Since $E$ satisfies $\G_{s}$, we have 
\[ \mu(W_{\p}) = \mu(E_{\p}) \le \dim(R_{\p}) + e-1.\]
Suppose that there exist ${\ds a_{1}, \ldots, a_{k-1} \in W}$ such that,  for all $0 \le i \le k-1$ and every subset ${\ds \{v_{1}, \ldots, v_{i}\} \subseteq \{1, \ldots, k-1\} }$, the statements (a),(b),(c) hold.

\medskip

\noindent Let $\PP_{k-1}$ be the power set of $\{1, \ldots, k-1\}$ and write ${\ds \widetilde{\nu}(i) \in \PP_{k-1}}$, where $i$ denotes the cardinality of the subset. For each ${\ds \widetilde{\nu}(i) = \{ v_{1}, \ldots, v_{i} \}}$, we write 
\[ N_{\widetilde{\nu}(i)} = \sum_{j=1}^{i} Ra_{v_{j}}, \;\; \mbox{and} \;\; \FF_{t}^{\widetilde{\nu}(i)} = \mbox{Fitt}_{t}(W/ N_{\widetilde{\nu}(i)}),  \;  \mbox{where} \; 0 \le t \le s+e-2.  \]
 Let 
\[ \mathcal{Q} = \mathcal{P} \cup \bigcup_{\tiny \widetilde{\nu}(i) \in \PP_{k-1}  } \bigcup_{t=0}^{s+e-2} \Min \left( \FF_{t}^{\widetilde{\nu}(i)}  \right).  \]
Then, by \cite[Lemma~1.3]{U94},  there exists ${\ds a_{k} \in W}$ such that for all ${\ds \widetilde{\nu}(i) \in \PP_{k-1}}$, 
\[ \mu \left(  \left( W/( N_{\widetilde{\nu}(i)} + Ra_{k}  )  \right)_{\p}   \right) = \max \left\{ 0, \;   \mu \left(  \left( W/N_{\widetilde{\nu}(i)}    \right)_{\p}   \right)  -1  \right\}  \;\; \mbox{for all} \;\;  \p \in \mathcal{Q}. \]
In particular, for all ${\ds \p \in \mathcal{P} \subseteq \mathcal{Q}}$,  we have
\[ \mu \left( \left( W/ \sum_{j=1}^{k} Ra_{j} \right)_{\p} \right)  =  \max \left\{ 0, \;   \mu \left(  \left( W/ \sum_{j=1}^{k-1} Ra_{j}  \right)_{\p}   \right)  -1  \right\} 
=    \max \left\{0, \, \mu(W_{\p} )-k \right\},  \]
where the last equality holds by the induction hypothesis. This proves statement (a).

\medskip

\noindent Let $\PP_{k}$ be the power set of $\{1, \ldots, k\}$  and write  ${\ds \nu(i) \in \PP_{k} }$, where $i$ denotes the cardinality of the  subset. We want to show that the statements (b), (c) hold for all $0 \le i \le k$ and for all ${\ds \nu(i)  = \{v_{1}, \ldots, v_{i}\}    \in \PP_{k} }$. If $k \notin \nu(i)$, then  ${\ds \nu(i) = \widetilde{\nu}(i) \in P_{k-1} }$. Thus, the claims follow from the induction hypothesis. Now, assume that  $k \in \nu(i)$ and note that  ${\ds \nu(i) = \widetilde{\nu}(i-1) \cup \{ k \} }$. Write
\[  N_{\nu(i)} = \sum_{j=1}^{i-1} Ra_{v_{j}} + Ra_{k} =  N_{\widetilde{\nu}(i-1)} + Ra_{k}. \]
\noindent To prove (b), let  $\p \in \Supp(E)$ with ${i-e+1 \le \dim(R_{\p}) \le s-1}$. Let ${\ds h= \dim(R_{\p}) +e-1-i }$. 
Suppose ${\ds \p \notin V \left( \FF_{h}^{\widetilde{\nu}(i-1)} \right) }$. Then 
${\ds \mu \left( \left( W/ N_{ \widetilde{\nu}(i-1) } \right)_{\p}  \right) \le h }$ and thus ${\ds  \mu \left( \left( W/ N_{\nu(i) } \right)_{\p} \right)  \le h }$ as well.

\medskip

\noindent Hence we may assume  ${\ds \p \in V \left( \FF_{h}^{\widetilde{\nu}(i-1)} \right) }$.  
For all prime ideals $\q$ with ${\ds \h(\q) \le i-e-1}$, by the induction hypothesis 
$\left( W/ N_{ \widetilde{\nu}(i-1) } \right)_{\q} =0$ and thus $ q \notin V \left( \FF_{h}^{\widetilde{\nu}(i-1)} \right)$. For all prime ideals $\q$ with ${\ds i-e \le \h(\q) \le h-e+i }$, by the the induction hypothesis 
\[ \mu \left( \left( W/ N_{ \widetilde{\nu}(i-1) } \right)_{\q} \right) \le \h(\q) + e-i \le h,\]
which implies that $ q \notin V \left( \FF_{h}^{\widetilde{\nu}(i-1)} \right)$.
That is, for all prime ideals $\q$ with ${\ds \h(\q) \le h-e+i }$,  we have ${\ds \q \notin V \left( \FF_{h}^{\widetilde{\nu}(i-1)} \right)}$.
Therefore, 
\[  h-e+i+1 \le  \h \left(  \FF_{h}^{\widetilde{\nu}(i-1)} \right) \le \h(\p) = h-e+i+1.\] 
Thus, ${\ds \p \in \Min \left(  \FF_{h}^{\widetilde{\nu}(i-1)} \right) \subseteq \mathcal{Q} }$ and we obtain the following.
\[ \begin{array}{rcl}
{\ds  \mu \left(  \left( W/ N_{{\nu}(i)}   \right)_{\p}   \right) } &=& {\ds   \max \left\{ 0, \;   \mu \left(  \left( W/( N_{\widetilde{\nu}(i-1)}   )  \right)_{\p}   \right)  -1  \right\}  } \\ && \\
&\le& {\ds \max \left\{0, \;  \h(\p) +e- i -1 \right\}  = h. }
\end{array}  \]

\medskip

\noindent To prove (c), let $\p$ be a prime ideal with $\h(\p) \le i-e$.  If ${\ds \p \notin V \left( \FF_{0}^{\widetilde{\nu}(i-1)}\right) }$, then 
${\ds \left( W/ N_{ \widetilde{\nu}(i-1) } \right)_{\p}  =0 }$.  Thus, ${\ds \left( W/ N_{\nu(i) } \right)_{\p}  =0 }$ as well.
Suppose then that  ${\ds \p \in V \left( \FF_{0}^{\widetilde{\nu}(i-1)} \right) }$.  For all primes $\q$ with $\h(\q) \le i-1-e$, by the induction hypothesis,  we have
\[  \left( W/ N_{ \widetilde{\nu}(i-1) } \right)_{\q}  =0, \] 
which implies $\q \notin V \left( \FF_{0}^{\widetilde{\nu}(i-1)} \right)$.
Therefore,
\[ i - e \le  \h( \FF_{0}^{\widetilde{\nu}(i-1)}  ) \le \h(\p) \le i-e, \]
and thus, ${\ds \p \in \Min \left( \FF_{0}^{\widetilde{\nu}(i-1)}  \right)  \subseteq \mathcal{Q} }$.  Moreover, we may assume ${\ds \p \in \Supp(E)}$. Since $\h(\p)= i-e$, the induction hypothesis implies
\[ \mu \left( \left( W/ N_{\widetilde{\nu}(i-1)}  \right)_{\p} \right)     -1   \le \dim(R_{\p}) + e-1-(i-1) -1  = (i-e) + e-i-1  = -1.\]
Therefore,
\[  \mu \left(  \left( W/ N_{{\nu}(i)}   \right)_{\p}   \right) =   \max \left\{ 0, \;   \mu \left(  \left( W/( N_{\widetilde{\nu}(i-1)}   )  \right)_{\p}   \right)  -1  \right\}  =0. \qedhere  \]
\end{proof}

\medskip

The following proposition allows us to obtain Theorem~\ref{exresmod} in the non local case.

\begin{Proposition}\label{exresmod1}
Let $R$ be a Noetherian ring. 
Let $\mathcal{P}$ be a finite {\rm(}possibly empty{\rm)} subset of $\Spec(R)$.
Let  $E$ be a finitely generated  $R$--module having a rank $e >0$.
Let $s$ be an integer.
Let $W$ be a submodule of $E$ such that  ${\ds \h( W :_{R} E) \ge s}$. Suppose that $E$ satisfies $\G_{s}$. Then there exist elements ${\ds a_{1}, \ldots, a_{s}}$ in $W$ such that
 \begin{enumerate}[{\rm (1)}]
 \item ${\ds \mu \left( \left( W/ \sum_{j=1}^{s} Ra_{j} \right)_{\p} \right) = \max\{0, \, \mu(W_{\p}) -s \} }$, whenever ${\ds \p \in \mathcal{P}}${\rm ;}
 \item ${\ds \h \left(  \sum_{j=1}^{i} Ra_{j} :_{R} E \right)  \ge i-e+1 }$, whenever ${\ds 0 \le i \le s}$.
 \end{enumerate}
In particular, if $s \ge e$, then  there exists an $i$--residual intersection of $E$ for every ${\ds e \le i \le s}$.
\end{Proposition}

\begin{proof}
The proof is almost identical to that of Theorem~\ref{exresmod}. We use induction on $k$, where $0\le k \le s$, to construct elements ${\ds a_{1}, \ldots, a_{k} \in W}$ such that  for all $0 \le i \le k$ we have  
\begin{enumerate}[(a)]
\item ${\ds \mu \left( \left( W/ \sum_{j=1}^{k} Ra_{j} \right)_{\p} \right) = \max\{0, \, \mu(W_{\p} )-k \} }$ whenever ${\ds \p \in \mathcal{P}}${\rm ;}
\item ${\ds \mu \left(  \left( W/\sum_{j=1}^{i} Ra_{j}   \right)_{\p} \right) \le  \dim(R_{\p}) +e-1-i}$ for all $\p \in \Supp(E)$ with ${i-e+1 \le \dim(R_{\p}) \le s-1}${\rm ;}
\item ${\ds  \left( W/\sum_{j=1}^{i} Ra_{j}  \right)_{\p} =0 }$ for all prime ideals $\p$ with $\h(\p) \le i-e$.
\end{enumerate}
The case $k=0$ is clear. Suppose that there exist ${\ds a_{1}, \ldots, a_{k-1} \in W}$ such that for all $0 \le i \le k-1$, the statements (a),(b),(c) hold.
We write 
\[ N= \sum_{j=1}^{k-1} Ra_{j}, \;\; \mbox{and} \;\; \FF_{t} = \mbox{Fitt}_{t}(W/ N),  \;  \mbox{where} \; 0 \le t \le s+e-2.  \]
 Let  ${\ds \mathcal{Q} = \mathcal{P} \cup  \bigcup_{t=0}^{s+e-2} \Min \left( \FF_{t}  \right) }$. 
Then, by \cite[Lemma~1.3]{U94}, there exists ${\ds a_{k} \in W}$ such that 
\[ \mu \left(  \left( W/( N+ Ra_{k}  )  \right)_{\p}   \right)  = \max \left\{ 0, \;   \mu \left(  \left( W/N    \right)_{\p}   \right)  -1  \right\}  \;\; \mbox{for all} \;\;  \p \in \mathcal{Q}. \]
The assertions follow as in the proof of Theorem~\ref{exresmod}. 
\end{proof}

One can even improve the previous result and choose the elements $a_1, \ldots, a_k$ as in Proposition~\ref{exresmod1} to be part of a generating sequence for the module $W$.

\begin{Corollary}\label{exresmod2}
Let $(R, \m, k)$ be a Noetherian local ring. 
Let  $E$ be a finitely generated  $R$--module having a rank $e >0$.
Let $s$ be an integer with $s \ge e$.
Suppose that $E$ satisfies $\G_{s}$.
Let $W$ be a submodule of $E$ such that ${\ds \h(W:_RE) \ge s \ge \mu(W)}$. 
\begin{enumerate}[{\rm (1)}]
\item There exists a generating sequence ${\ds a_{1}, \ldots, a_{s}}$ of $W$ such that
\[ \h \left(   \sum_{j=1}^{i} Ra_{j}  :_{R} E \right)  \ge i-e+1 \]
for each $0 \le i \le s$.
\item If $k$ is infinite, then there exists a generating sequence ${\ds a_{1}, \ldots, a_{s}}$ of $W$ such that
\[ \h \left(   \sum_{j=1}^{i} Ra_{v_j}  :_{R} E  \right) \ge i-e+1 \]
for each $0 \le i \le s$ and for every subset ${\ds \{v_{1}, \ldots, v_{i}\} \subseteq \{1, \ldots, s\} }$.
\end{enumerate}
\end{Corollary}

\begin{proof}
It follows from applying Theorem~\ref{exresmod} and Proposition~\ref{exresmod1} with ${\ds \mathcal{P} = \{ \m \} }$.
\end{proof}

\begin{Remark}\label{rired}{\rm
Let $(R, \m)$ be a Noetherian local ring with infinite residue field. 
Let $E$ be a finitely generated $R$-module with a rank $e$ and analytic spread $\ell$.
Suppose $E$ admits a proper reduction and let  $U$ be a minimal reduction of $E$. 
If $E_{\p}$ is of linear type for all primes $\p$ with $\h(\p) \le \ell -e$, then 
$(U:_{R} E)$ is an $\ell$--residual intersection of $E$.  
}\end{Remark}

The previous remark gives a concrete strategy to build $\ell$-residual intersections of a module of analytic spread $\ell$ from its proper minimal reductions. 
In the case of ideals, residual intersections obtained as colon ideals as in Remark~\ref{rired} have proved to be extremely useful in the study of the Cohen-Macaulay property of blowup algebras \cite{{JU96},{U96},{PX13},{M15}} and of the core \cite{{CPU01},{CPU02}}. 
It is then natural to ask what classes of modules are of linear type locally in codimension at most $\ell - e$.

\begin{Proposition}\label{llt}
Let $R$ be a Cohen-Macaulay local ring of dimension $d$ with infinite residue field. Let $E$ be a finite, torsionfree $R$-module with rank $e>0$,  and analytic spread $\ell$. 
  \begin{enumerate}[{\rm (1)}] 
  \item If $E$ has the projective dimension $1$, then  $E$ is of linear type locally in codimension at most $\ell - e$.
  \item If $\ell=e$, then $E$ is of linear type locally in codimension $0$.
  \item If $E$ satisfies $\G_{2}$, then $E$ is of linear type locally in codimension at most $1$.
  \item If $d =3$ and $E$ satisfies $\G_{\ell -e +1}$, then $E$ is of linear type locally in codimension at most $\ell -e$.
 \end{enumerate} 
\end{Proposition}

\begin{proof} Statement (1) follows from \cite[Proposition~4]{Avramov}. Statements (2) and (3) follow from the definition of rank and the $\G_{2}$ assumption, respectively.  

\medskip

\noindent (4): Note that ${\ds e \le \ell \le 3+e-1 = e +2}$, by \cite[Proposition~2.3]{SUV03}.  If $\ell=e$, then the assertion follows from (2). If $\ell = e+1$, then the claim follows from (3).
Now suppose $\ell = e+2$. Then $E$ satisfies $\G_{3}$. In particular, $E$ satisfies $\G_{2}$ and $E_{\p}$ is free for every prime $\p$ with $\h(\p) \le 1$. Let $\p$ be a prime with $\h(\p)=2$. Then $E_{\p}$ satisfies $\G_{\infty}$, so it is of linear type by \cite[Proposition~4.4(b)]{SUV03}.
 \end{proof}
 
 As seen in Proposition~\ref{llt}, modules of projective dimension $1$ are of linear type locally in codimension at most $\ell-e$,  a fact that was essential in the proof of \cite[Theorem~2.4]{CPU03}. The next proposition gives two classes of modules that are in fact free locally in codimension at most $\ell-e$. 
 
\begin{Proposition}\label{IdealModules}
 Let $R$ be a Cohen-Macaulay local ring of dimension $d \ge 3$ with infinite residue field. Let $I$ be an ideal of height $g \ge 2$ and suppose that $\,g +1 = \ell(I) < \mu(I)$. Then, for $e \ge 2$, the modules 
   $$ E = I \oplus R^{e-1} \quad \mathrm{ and } \quad M = \underbrace{I \oplus \ldots \oplus I}_{e \; times}$$
 have analytic spread $\ell = \ell(I)+ e -1$, are free locally in codimension at most $\ell - e$ and admit a proper minimal reduction. 
\end{Proposition}

\begin{proof}
  Notice that, for a prime ideal $\p$, $E_{\p}$ is free if and only if $\p \notin V(I)$. Hence, the codimension of the non-free locus of $E$ is $g$. Similarly, the codimension of the non-free locus of $M$ is $g$. 
  Moreover, by \cite[Remark~3.9]{BAM20} it follows that $\ell(E)= \ell(I) + e -1$, while \cite[Corollary~3.13]{BAM20} implies that $\ell(M)= \ell(I) + e -1$. 
  Let $\ell = \ell(E) = \ell(M)$. Then, $g = \ell(I) -1 = \ell - e$, and the first two assertions are proved.
  
  \medskip
  
 \noindent To show that $E$ and $M$ admit a proper minimal reduction, it suffices to prove that $\mu(E) > \ell$ and $\mu(I) > \ell$. This follows from the assumption that $\mu(I) > \ell(I)$, since 
   $\mu(E) = \mu(I) + e - 1 > \ell(I) + e -1 = \ell\,$ and $\,\mu(M) = e \,\mu(I) = \mu(I) + (e -1)\mu(I) > \ell(I) + e -1 = \ell$.  
\end{proof}

  \begin{Example}
  {\rm Let $R= k[x_1, x_2, x_3, x_4]$ be a polynomial ring over an infinite field $k$ with $\dim R=4$. 
  Let $I= (x_1x_{2}, x_{2}x_{3}, x_{3}x_{4}, x_{1}x_{4})$ be the edge ideal of a square.   
  Then, $\h(I) = 2$ and $\ell(I) = 3 = \mu(I)-1$ by \cite[Proposition~3.1]{Vil}. Hence, $I$ satisfies the assumptions of Proposition~\ref{IdealModules}.}
 \end{Example}

\noindent  In Section~\ref{CoreM}, we consider another class of modules which are of linear type locally in codimension at most $\ell - e$. 

\medskip

\section{Artin--Nagata Properties}\label{SectionAN}

In this section we investigate when the residual intersections of a module are Cohen-Macaulay. This has been a central theme in the literature on residual intersections of ideals \cite{{AN72},{HVV85},{Huneke83},{U94},{CEU01}}. If the residual intersections of $I$ are Cohen-Macaulay, they can often be used to understand the powers of $I$ and its minimal reductions \cite{{U94},{JU96},{U96}}. Similarly as in the case of ideals \cite[Definition~1.2]{U94}, we define the Artin-Nagata conditions for modules as follows. 

\begin{Definition}\label{ANmod}{\rm
Let $R$ be a Noetherian ring, $E$ a finitely generated torsionfree $R$--module having a rank $e>0$, and $s$ an integer with $s \ge e$.
We say that $E$ satisfies ${\ds \AN_{s}}$ if $R/K$ is Cohen--Macaulay for every $i$--residual intersection $K$ of $E$ and 
for every ${\ds e \le i \le s}$.
}\end{Definition}

It is well-known that every strongly Cohen-Macaulay ideal in a Gorenstein local ring satisfies ${\ds \AN_{s}}$ for every integer $s$ \cite[Remark~2.10]{U94}. In particular, this class includes all perfect ideals of height two, which can be regarded as torsionfree modules of rank one and projective dimension one. Our next result shows that the same is true for all torsionfree modules of projective dimension one, independently of their rank.

\begin{Theorem}\label{pd1}
Let $R$ be a Cohen--Macaulay ring. Let $E$ be a finitely generated,
torsionfree $R$--module with rank $e>0$, and projective dimension $1$.
 Then $E$ satisfies ${\ds \AN_{s}}$ for any $s \ge e$. 
\end{Theorem}

\begin{proof} Let $e \le i \le s$ and let $K$ be an $i$--residual intersection of $E$. Then there exists $i$--generated submodule $U \subseteq E$ such that $K=(U:_R E)$, and $\h(K) \ge i-e+1$. Consider the
commuting diagram
\[\begin{CD}
0 @>>> R^{n-e} @>{\varphi}>> R^n @>>> E @>>> 0 \\
&&&& @AA{\psi}A @AAA\\
&&&& R^{i} @>>> U @>>> 0\\
&&&&&& @AAA\\
&&&&&& 0
\end{CD}\]whose rows are exact. Then there is an exact sequence
\[ R^{n-e+i} \stackrel{\Psi}{\longrightarrow} R^n \longrightarrow E/U \longrightarrow 0,\]
where $\Psi=[\varphi \,|\, \psi]$. Since $K=\ann(E/U)$ and
$\mbox{\rm Fitt}_0(E/U)$ have the same radical, $\h(K)=\h(\mbox{\rm Fitt}_0(E/U))$. Therefore we obtain
\[ i-e+1 \le \h(K) = \h(\mbox{\rm Fitt}_0(E/U))=\h(I_n(\Psi)) \le
 (n-n+1)(n+i-e-n+1)=i-e+1, \] where the last inequality holds according to the Eagon-Northcott
bound. By \cite{BE77}, $R/K$ is perfect. Notice that
$\pdim(R/K)=\grade(R/K)=\grade(K)=\h(K)=i-e+1 < \infty$. Therefore $R/K$
is Cohen--Macaulay \cite[Theorem~2.1.5]{BH}. 
\end{proof}

If $I$ is a perfect ideal of height two satisfying $G_{\ell(I)}$, its Artin-Nagata property crucially allows to characterize when $r(I) \le \ell(I) -1$ \cite[Corollary~5.3]{U96}. By \cite[Theorem~2.6 and Corollary~3.4]{CPU02}, this is equivalent to the balanced condition that $\core(I) = (J :_{R} I) I =  (J :_{R} I) I$ for every minimal reduction $J$ of $I$ as described in Theorem~\ref{TheoremCPU02}. In fact, in the setting of Theorem~\ref{TheoremCPU02}, $(J :_R I)$ is an $\ell(I)$-residual intersection of $I$ for any minimal reduction $J$ of $I$. 

\medskip

Similarly, if $E$ has projective dimension $1$ and satisfies $\G_{\ell(E) -e +1}$, then by \cite[Theorem~2.4]{CPU03}  the condition that $\r(E) \le \ell(E) - e$  is equivalent to  $\core(E) = (U :_{R} E) E =  (U :_{R} E) U$ for every minimal reduction $U$ of $E$ as recalled in Theorem~\ref{TheoremCPU03}. The result is proved by induction on the rank, the rank $1$ case being the known case of perfect ideals of height two \cite[Theorem~2.6 and Corollary~3.4]{CPU02}. It is then natural to ask whether Theorem~\ref{TheoremCPU03} can be reinterpreted in terms of the residual intersections of $E$.

\medskip

A careful analysis of the proof of \cite[Theorem~2.4]{CPU03} reveals that, in order for the induction to be possible, it suffices that $\h (U :_{R} E) \ge \ell(E) - e +1$ for any minimal reduction $U$ of $E$ \cite[Lemma~2.2(b)]{CPU03}. According to Definition~\ref{resmod}, this condition means that $(U :_{R} E)$ is an $\ell(E)$-residual intersection of $E$. By Remark~\ref{rired}, this is guaranteed whenever $E$ is of linear type locally in codimension at most $\ell(E) -e$, which follows from \cite[Propositions~3~and~4]{Avramov} if $E$ has projective dimension one. On the other hand, it is reasonable to ask whether the proof of \cite[Theorem~2.4]{CPU03} could be generalized to other classes of modules that are of linear type locally in codimension at most $\ell(E) -e$.

\medskip

With Theorem~\ref{TheoremCPU02} in mind, we should aim at a class of modules that generalizes universally weakly $(\ell(I)-1)$-residually $(\S_{2})$ ideals \cite{{CEU01},{CPU02}}. We omit this definition, as it is technical and will not be needed in the rest of the paper. However, we recall that in a Gorenstein ring, in order for an ideal $I$ satisfying $\G_{\ell(I)}$ to be universally weakly $(\ell(I)-1)$-residually $(\S_{2})$, it suffices that ${\ds \Ext_R^{g+j}(R/I^j,R)= 0}$ for ${\ds 1 \le j \le \ell(I)-g}$ \cite[Theorem 4.1]{CEU01}, \cite[page 2582]{CPU02}. 
When $g=2$, the latter condition can be  rewritten as ${\ds \Ext_R^{\,j+1}(I^j,R) \simeq \Ext_R^{j+2}(R/I^j,R)= 0}$ for ${\ds 1 \le j \le \ell(I)-2 }$. In the next section, we prove a module version of Theorem~\ref{TheoremCPU02} for modules satisfying ${\ds \Ext_R^{\,j+1}(E^j,R)= 0}$ for ${\ds 1 \le j \le \ell(E)-e -1}$.

\section{Core of Modules}\label{CoreM}

In this section, we extend Theorem~\ref{TheoremCPU02} (or \cite[Theorem~2.4]{CPU03}) to orientable modules. Recall that a finitely generated $R$--module $E$ is called {\em orientable} if it has a rank $e >0$ and
${\ds ( \wedge^{e} E)^{**}  \simeq R }$, where ${\ds (-)^{*}}$ denotes the functor ${\ds \Hom_{R}( - ,  R) }$. Examples of orientable modules include modules of finite projective dimension over any Noetherian ring, and finitely generated modules over a UFD. 

\medskip

 When studying the Rees algebra of a finitely generated, torsionfree module $E$ with a positive rank, one can construct a Bourbaki ideal $I$ associated to $E$, which shares many of the properties of the module. For instance, the Rees algebra of $E$ is Cohen-Macaulay if and only if the Rees algebra of $I$ is Cohen-Macaulay \cite[Theorem~3.5]{SUV03}. Moreover, ${\rm{grade}} I \ge 2$ if and only if $E$ is orientable \cite[Proposition~3.2]{SUV03}. 

\medskip

The key idea underlying the results of \cite{SUV03} is to identify ideals as torsionfree modules of rank one and perform an induction on the rank, to reduce a statement about modules to a statement about ideals. However, this process requires that one constructs a Bourbaki ideal by first enlarging the ring $R$ to a generic extension $R'$. While a generic extension of the ring is not expected to preserve the core, in \cite{CPU03} Corso, Polini and Ulrich could similarly reduce the problem of studying the core of a module of projective dimension one to the case of perfect ideals of height two, by going modulo general elements rather than generic elements. In a similar fashion, an induction on the rank will allow us to extend Theorem~\ref{TheoremCPU02} to orientable modules.

\medskip

We begin by identifying a sufficient condition for an orientable module $E$ to be of linear type locally in codimension at most $\ell-e$.

\begin{Lemma}\label{htUEExt}
Let $R$ be a Gorenstein local ring of dimension $d \ge 4$. Let $E$ be a finite, torsionfree, and orientable $R$-module with $\rank(E)=e>0$, and analytic spread $\ell(E)=\ell \ge e+2$. Assume that $E$ is $\G_{\ell-e+1}$ and that ${\ds \Ext_R^{\,j+1}(E^j,R)= 0}$ for ${\ds 1 \le j \le \min \{\ell-e-1, \, d-3 \} }$. Then
\begin{enumerate}[{\rm (1)}]
\item  ${\ds E_{\p}}$ is of linear type for every prime $\p$ with $\h(\p) \le \ell-e$.
\item  For any minimal reduction $U$ of $E$, we have ${\ds  \h(U:_{R} E) \ge \ell -e +1}$.  
\end{enumerate}
\end{Lemma}

\begin{proof} Statement (2) follows from (1) directly.  To prove (1), let $\p$ be  a prime with ${\ds \h(\p) \le \ell - e}$.  Then ${\ds E_{\p}}$ is finite, torsionfree, orientable, with $\rank(E_{\p})=e$.   Note that $E$ is $\G_{\ell-e+1}$, where $\ell - e+ 1 \ge 3$, and ${\ds E_{\p}}$ satisfies  $\G_{\infty}$.  If $\h(\p)=0$, then $E_{\p}$ is free because $E$ has a rank. If $\h(\p)=1$, then $E_{\p}$ is free because $E$ satisfies $\G_{2}$.  If $\h(\p)=2$ or $\h(\p)=3$, then $E_{\p}$ is of linear type because $E_{\p}$ is orientable and satisfies $\G_{\infty}$ \cite[Proposition~4.4~(b), Proposition~4.6~(a)]{SUV03}.

\medskip

\noindent Suppose that ${\ds  \dim( R_{\p} )  \ge 4}$ and that $E_{\p}$ is not free.  
Suppose that ${\ds \ell(E_{\p}) = e +1}$. By \cite[Proposition~4.1(b)]{SUV03}, it must then be that ${\ds \mu(E_{\p}) \le e+1}$,  whence ${ \ds \ell(E_{\p}) = e +1 = \mu(E_{\p}) }$. Thus, a generic Bourbaki ideal $I$ of ${ \ds E_{\p} }$ satisfies $\, {\ds \mu(I) = \ell(I) = 2 = \h(I)}$ and so is a $\G_{\infty}$ complete intersection.  By \cite[Theorem~6.1]{HSV83}, it follows that $I$ is of linear type. By \cite[Theorem~3.5]{SUV03}, ${\ds E_{\p}}$ is of linear type. 

\medskip

\noindent We may now assume that $\ell(E_p)\ge e+2$.  Let ${\ds m=\min \left\{ \ell(E_{\p}) - e -1, \; \dim(R_{\p})-3   \right\} }$. Then $m \ge 1$. 
Since
\[ \ell(E_{\p}) - e -1 \le \ell(E) - e-1, \quad \mbox{and} \quad  \dim(R_{\p})-3  \le \ell-e-3 < \ell -e-1, \]
then $m \le \ell (E) -e-1$. Also, $m \le \dim(R_{\p})-3  \le d-3$ and hence
\[ m \le \min \{\ell-e-1, \, d-3 \}. \]
Thus, for every ${\ds 1 \le j \le m}$, we have 
\[  \Ext_{R_{\p}}^{\,j+1}(E_{\p}^j,R_{\p})  = \left( \Ext_R^{\,j+1}(E^j,R)  \right)_{\p}  = 0.     \]
By \cite[Theorem~3.3]{C2}, $E_{\p}$ is of linear type. 
\end{proof}

\medskip

\begin{Proposition}\label{free}
Let $R$ be a Gorenstein local ring of dimension $d$ with infinite residue field. Let $E$ be a finite, torsionfree, and orientable $R$-module with $\rank(E)=e>0$, and analytic spread $\ell(E)=\ell \ge e+2$. Assume that $E$ is $\G_{\ell-e+1}$ and that ${\ds \Ext_R^{\,j+1}(E^j,R)= 0}$ for every ${\ds 1 \le j \le \ell-e-1}$.  Suppose that the Rees algebra $\R(E)$ satisfies $(\S_{2})$.   Then for any minimal reduction $U$ of $E$, we have
\[U/(U:_{R} E)U \simeq (R/ (U:_{R} E) )^{\ell}.\]
\end{Proposition}

\begin{proof}
We proceed by induction on $e$. Suppose $e=1$. Then $E$ is isomorphic to an ideal $I$ with analytic spread $\ell \ge 3$ and height $g\ge 2$, since $E$ is orientable \cite[Proposition~3.2]{SUV03}.  Suppose $g \ge 3$.  Then ${\ds 1 \le g -2 \le \ell -2}$ and by assumption we have
\[ 0 = \Ext_R^{\,g-1}(I^{g-2},R)  \simeq  \Ext_{R}^{g}( R/ I^{g-2}, R) \neq 0, \]
because  ${\ds \grade (I^{g-2}) = \grade(I)=g}$, which is a contradiction. Thus, $g=2$. 

\medskip

\noindent Moreover, $I$ satisfies  $\G_{\ell-1}$ with $\ell -1 \ge 2$ and 
\[ \Ext_R^{\,j+2}(R/I^j,R) \simeq \Ext_R^{\,j+1}(I^j,R)= 0 \;\; \mbox{for} \;\; 1 \le j \le \ell-2. \]
Hence, by \cite[Theorem~4.1]{CEU01},  $I$ is $(\ell-1)$--residually $(S_{2})$. Moreover, $I$ satisfies $\G_{\ell}$ (by assumption) and is $(\ell-2)$--weakly residually $(S_{2})$. The claim  now follows from \cite[Lemma~2.5]{CPU02}.

\medskip

\noindent Suppose ${\ds e \ge 2}$ and  the statement is true for modules satisfying the assumptions and having rank $e-1$.  Let $U$ be a minimal reduction of $E$, $x$ a general element of
$U$, and let $\widetilde{E}$ and $\widetilde{U}$ denote $E/Rx$ and  $U/Rx$, respectively.  Then by \cite[Lemma~2.2~(a),(d),(e)]{CPU03}, ${\ds \widetilde{U}}$ is a minimal reduction of ${\ds \widetilde{E}}$, ${\ds \rank( \widetilde{E} ) = e-1}$, and ${\ds \ell( \widetilde{E} ) = \ell -1}$. 
By Lemma~\ref{htUEExt}, we have ${\ds  \h(U:_{R} E) \ge \ell -e +1}$. Thus, by \cite[Lemma~2.2~(b)]{CPU03}, $\widetilde{E}$ is torsionfree and satisfies $\G_{\ell-e+1}$. 
Moreover, since $Rx$ is free of rank $1 <e$, we have ${\ds (\wedge^{e}Rx)^{**} =0 }$. Because  $E$ is orientable,  the exact sequence 
${\ds 0 \to Rx \to E \to \widetilde{E} \to 0 }$ implies that $\widetilde{E}$ is orientable as well.

\medskip

\noindent  Since $x$ is regular on $\mathcal{R}(E)$, by \cite[Lemma~2.1]{CPU03} we have the following exact sequence.
\[  0 \rar E^{j-1} \stackrel{\cdot x}{\lar} E^{j} \lar E^{j}/xE^{j-1} \rar 0. \]
This induces the long exact sequence:
\[ \cdots \lar \Ext^{j}_{R}( E^{j-1}, R) \lar \Ext^{j+1}_{R}( E^{j}/xE^{j-1}, R) \lar \Ext^{j+1}_{R}( E^{j}, R) \lar  \Ext^{j+1}_{R}( E^{j-1}, R)  \lar \cdots. \]
By assumption, we have ${\ds  \Ext^{j}_{R}( E^{j-1}, R)  =0 = \Ext^{j+1}_{R}( E^{j}, R)}$. Thus, ${\ds \Ext^{j+1}_{R}( E^{j}/xE^{j-1}, R) =0}$. Since the Rees algebra $\R(E)$ satisfies $(\S_{2})$, by \cite[Lemma~2.2]{CPU03}, we have ${\ds  \R(\widetilde{E}) \simeq  \R(E)/(x)}$ and hence ${\ds \Ext_R^{\,j+1}( \widetilde{E}^j,R) =  \Ext^{j+1}_{R}( E^{j}/xE^{j-1}, R) = 0}$ for ${\ds 1 \le j \le  \ell( \widetilde{E} )  - \rank( \widetilde{E} ) -1=   \ell-e-1 }$.
Then by the induction hypothesis, we have 
\[ \widetilde{U}/( \widetilde{U} :_{R} \widetilde{E} ) \widetilde{U} \simeq (R/ ( \widetilde{U} :_{R} \widetilde{E}) )^{\ell -1}.\] 

\medskip

\noindent Let ${\ds S= R/(\widetilde{U} :_{R} \widetilde{E} ) = R/(U:_RE) }$. The induction hypothesis implies that ${\ds \widetilde{U} \otimes S}$ is a free $S$--module of rank $\ell -1$. We want to show that $U \otimes S$ is a free $S$--module of rank $\ell$. 
Equivalently, we want to show that 
\[ \FF_{\ell -1}(U \otimes S) =0, \quad \mbox{and} \quad \FF_{\ell}(U \otimes S) = S. \]
Note that  the ideals of minors commute with ring extensions. i.e., ${\ds \FF_{\ell -1}(U \otimes S) = \FF_{\ell -1} (U) \otimes S  }$.  We can choose a presentation ${\ds R^{t} \stackrel{\varphi}{\lar} R^{\ell} \stackrel{\pi}{\lar} U \rar 0}$, where  ${\ds \FF_{\ell}(U \otimes S) = S}$ is satisfied automatically. Thus, the assertion follows once we prove that ${I_{1}(\varphi) \subseteq (U:_RE)}$. Now this is equivalent to 
\[ \Big( (Rx_{1} + \cdots+ \widehat{Rx_{i}} + \cdots + Rx_{\ell}) :_R x_{i}  \Big) \subseteq (U:_RE) \;\; \mbox{for every} \;\; 1 \le i \le \ell, \] 
where $\{x_1, \ldots, x_{\ell}\}$ is a generating set of $U$.

\medskip

\noindent Take ${\ds x \in \{ x_{1}, \ldots, \widehat{x_i}, \ldots, x_{\ell} \} }$ and let ${\ds L_{i} = (Rx_{1} + \cdots+ \widehat{Rx_{i}} + \cdots + Rx_{\ell})}$.  Then we have
\[   ( L_{i} :_R x_{i}  ) =  \Big( \widetilde{ L_{i}} : _R\widetilde{x_{i}} \Big)  \subseteq ( \widetilde{U} :_R \widetilde{E} ) =  (U :_R E),       \]
where the middle inclusion holds by the induction hypothesis.
\end{proof}

\medskip

\noindent Now we are ready to prove our main result. In particular, our result can be viewed as generalization of \cite[Theorem~2.4]{CPU03} to a class of modules that satisfy certain Ext vanishing conditions.

\begin{Theorem}\label{core}
Let $R$ be a Gorenstein local ring of dimension $d$ with infinite residue field. Let $E$ be a finite, torsionfree, and orientable $R$-module with $\rank(E)=e>0$, and analytic spread $\ell(E)=\ell \ge e+2$. Assume that $E$ is $\G_{\ell-e+1}$ and that ${\ds \Ext_{R}^{\,j+1}(E^j, R)= 0}$ for every ${\ds 1 \le j \le \ell-e-1}$. Suppose that the Rees algebra $\R(E)$ satisfies $(\S_{2})$.
The following are equivalent{\rm \,:}

\begin{enumerate}[{\rm (i)}]
\item $(U:_RE)E \subseteq \core(E)$ for every minimal reduction $U$ of $E${\rm \,;}
\item $(U:_RE)U=(U:_RE)E=\core(E)$ for every minimal reduction  $U$ of $E${\rm \,;}
\item$(U:_RE)$ is independent of $U$ for every minimal reduction $U$ of $E$.
\end{enumerate}
\end{Theorem}

\begin{proof}  (i)$\Rightarrow$(ii). It suffices to show that $\core(E) \subseteq (U:_RE)U$. We prove this by induction on $e$. Suppose $e=1$. As proved in Proposition~\ref{free}, $E$ is isomorphic to an ideal $I$ of height $2$. Let $S=R(X_{1}, \ldots, X_{n})$, where $X_{1}, \ldots, X_{n}$ are variables over $R$ and $n$ is any positive integer. Then $IS$ has height $2$ and satisfies $\G_{\ell-1}$ with $\ell -1 \ge 2$. Moreover,
\[ \Ext_{S}^{\,j+2}(S/(IS)^j,S) \simeq \Ext_{S}^{\,j+1}((IS)^j,S) = \Ext_{R}^{\,j+1}(I^{j}, R) \otimes S= 0 \;\; \mbox{for} \;\; 1 \le j \le \ell-2. \]
Hence, by \cite[Theorem~4.1]{CEU01},  $IS$ is $(\ell-1)$--residually $(S_{2})$. Therefore, $I$ is universally $(\ell -1)$--residually $(\S_{2})$. 
The assertion now follows from \cite[Theorem~2.6]{CPU02}.

\medskip

\noindent Suppose ${\ds e \ge 2}$ and  the result holds for all modules satisfying the assumptions with rank $e-1$.  Let $x$ be a general element of
$U$.  Let $\widetilde{E}$ and $\widetilde{U}$ denote $E/Rx$ and  $U/Rx$, respectively. Then, as proved in Proposition~\ref{free},  $\widetilde{E}$ is a finite, torsionfree, and orientable $R$-module with $\rank(\widetilde{E})=e-1>0$, and analytic spread $\ell(\widetilde{E})=\ell-1$. Also $\widetilde{E}$ satisfies $\G_{\ell-e+1}$ and that ${\ds \Ext_R^{\,j+1}( \widetilde{E}^j,R)= 0}$ for every ${\ds 1 \le j \le \ell-e-1}$, as seen in the proof of Proposition~\ref{free}.

\medskip

\noindent By the induction hypothesis, we have
\[\core( \widetilde{E}) \subseteq (\widetilde{U} :_R \widetilde{E}) \widetilde{U} = (U:_RE) \widetilde{U}. \]
Note that ${\ds \core(E) \subseteq \bigcap_{x \in V} V}$, where $V$ is a reduction of $E$.  Moreover,  by \cite[Lemma~2.2~(c)]{CPU03}, ${\ds \core( \widetilde{E}) = \bigcap_{x \in V} \widetilde{V}}$, where $V$ is a reduction of $E$.  Therefore, we have 
\[ (\core(E) + Rx)/Rx  \subseteq \core( \widetilde{E}) \subseteq (U:_RE) \widetilde{U}. \]
Hence ${\ds \core(E) \subseteq (U:_RE)U+Rx}$.  Let $x_{1}, \ldots, x_{\ell}$ be a minimal generating set of $U$. Replacing $x$ by $x_{i}$, we obtain
\[ \core(E) \subseteq \bigcap^{\ell}_{i=1} (Rx_{i}+(U:_RE)U). \]
It is enough to show that  ${\ds \bigcap^{\ell}_{i=1} (Rx_{i}+(U:_RE)U) \subseteq (U:_RE)U}$.  If $(U:_RE)=R$, then the result is obvious. Suppose  $(U:_RE) \neq R$.  Let ${\ds \alpha \in \bigcap^{\ell}_{i=1} (Rx_{i}+(U:_RE)U)}$. Then, for each $i=1, \ldots, \ell$, 
\begin{equation}
 \label{eqalpha}
   \alpha = r_{i} x_{i} + \beta_{i},
\end{equation} 
where $r_{i} \in R$ and $\beta_{i} \in (U:_RE)U$.  It suffices to show that ${\ds r_{i} \in (U:_RE)}$ for all $i$.  By Proposition~\ref{free}, $U/(U:_{R} E)U \simeq (R/ (U:_{R} E) )^{\ell}$. Let $\overline{r_i}$ and  $\overline{x_i}$ denote the images of $r_i$ and $x_i$ in $(R/ (U:_{R} E) )^{\ell}$, via the isomorphism $U/(U:_{R} E)U \simeq (R/ (U:_{R} E) )^{\ell}$. Since for all $1 \le j \le \ell$ we have ${\ds \overline{r_{j}} \cdot \overline{x_{j}} = \overline{\alpha}}$, it follows that for all $1 \le j \neq k \le \ell$ 
\[ \overline{r_{j}} \cdot \overline{x_{j}} = \overline{r_{k}} \cdot \overline{x_{k}}.\]
In particular, ${\ds \overline{r_{j}} \cdot \overline{x_{j}} - \overline{r_{k}}\cdot \overline{x_{k}} = \overline{0} }$. 
By Proposition~\ref{free},  the images ${\ds \overline{x_{1}}, \ldots, \overline{x_{\ell}}}$ form a basis of the free module $U/(U:_RE)U$ over $R/(U:_RE)$. By the linearly independence, we have ${\ds \overline{r_{j}} = \overline{r_{k}} = \overline{0} }$. That is, for all $i=1, \ldots, \ell$, ${\ds r_{i} \in (U:_RE)}$. 
This shows that
\[ \alpha = r_{i} x_{i} + \beta_{i} \in (U:_RE)U. \]

\medskip

\noindent  (ii)$\Rightarrow$(iii).  Let $U_1$ and $U_2$ be minimal
reductions of $U$. Since $(U_{1}:_RE)E=\core(E)=(U_{2}:_RE)E$, we have
$(U_{1}:_RE)E=(U_{2}:_RE)E \subseteq U_{2}$. Then $(U_{1}:_RE) \subseteq
(U_{2}:_RE)$. Similarly, $(U_{2}:_RE) \subseteq (U_{1}:_RE)$.

\medskip

\noindent (iii)$\Rightarrow$(i).  Clearly $(V:_RE)E \subseteq V$ for any
minimal reduction $V$ of $E$. But by the independence (i),
$(U:_RE)E=(V:_RE)E$ for any minimal reduction $U$ of $E$. Thus $(U:_RE)E
\subseteq V$ for all minimal reduction $V$ of $E$ and hence $(U:_RE)E
\subseteq \core(E)$.
\end{proof}

\medskip

Our next goal is to construct examples of classes of modules satisfying the assumptions of Theorem~\ref{core}. Before doing so, we include a technical lemma, whose statement is well known. We include the proof here for completeness.
 
\begin{Lemma} \label{torsion}
  Let $R$ be an integral domain and $M$ a finitely generated $R$-module. Let $\tau_{R}(M)$ be the torsion submodule of $M$, $X$ an indeterminate and let $R[X]$ be a polynomial ring over $R$.  Then the $R$-torsion submodule of $M \otimes_R R[X]$ is $\tau_R(M) \otimes_R R[X]$.
\end{Lemma}

\begin{proof}
 Notice that $M \otimes_R R[X] \simeq M[X]$. If ${\ds f(X) = \sum_{i=0}^n a_i X^i \in \tau_R (M[X]) }$, then there exists a non-zero element $c \in R$ so that $c f(X)=0$. This implies that $c a_i =0$ for all $i \ge 0$, i.e., $a_i \in \tau_R (M)$ for all $i$. 
 Conversely, let ${\ds f(X) = \sum_{i=0}^n a_i X^i \in M[X]}$ and suppose that $a_i \in \tau_R (M)$ for $0 \le i \le n$. Then, there exist non-zero elements $c_i \in R$ so that $c_i a_i = 0$, for all $0 \le i \le n$. Since $R$ is a domain, $c = c_0 \cdots c_n$ is a non-zero element of $R$ and $c f(X) =0$. Therefore, ${f(X) \in \tau_R (M[X])}$.
\end{proof}

The first class of modules we consider is certain types of ideal modules arising from height two ideals. 

\begin{Theorem} \label{balancedI}
  Let $R$ be a Gorenstein local domain of dimension $d$ with infinite residue field. Let $I$ be an $R$-ideal of height two, analytic spread $\ell(I) \ge 3$, and reduction number $r(I) \le \ell(I) -1$. Assume that $I$ satisfies $\G_{\ell(I)}$ and that $\depth(I^j) \ge d-j$ for ${\ds 1 \le j \le \ell(I)-1}$. Then, for $e>0$, the $R$-module $E = I \oplus R^{e-1}$ has reduction number $\r(E) = \r(I)$ and  $\rees{E}$ is Cohen-Macaulay. Moreover, the following are equivalent{\rm \,:}

\begin{enumerate}[{\rm (i)}]
\item $(U:_RE)E \subseteq \core(E)$ for every minimal reduction $U$ of $E${\rm \,;}
\item $(U:_RE)U=(U:_RE)E=\core(E)$ for every minimal reduction  $U$ of $E${\rm \,;}
\item$(U:_RE)$ is independent of $U$ for every minimal reduction $U$ of $E$.
\end{enumerate} 
\end{Theorem}

\begin{proof}
  It is clear that $E = I \oplus R^{e-1}$ is a torsionfree, orientable module of rank $e>0$ and analytic spread $\ell = \ell(I)+e-1\ge e+2$. Moreover, $E$ satisfies $\G_{\ell - e +1}$ since $I$ does. 
  To prove the remaining statements, notice that $\Sym(E) \simeq \Sym(I) \otimes_{R} \Sym(F)$, where $F=R^{e-1}$. Since $\Sym(F)$ is a polynomial ring in $e-1$ variables with coefficients in $R$ and $R$ is a domain, by Lemma~\ref{torsion} it follows that $\R(E) \simeq \R(I) \otimes_{R} \Sym(F)$.
   Hence, for all $j$ one has
  \[ E^j = [\R(E)]_{j} \simeq [\R(I) \otimes_{R} \Sym(F)]_{j} = \bigoplus_{0 \le i \le j} I^i \otimes_{R} F^{j-i}. \]
 Since $R$ is Gorenstein,  the assumption that $\depth(I^j) \ge d-j$ for ${\ds 1 \le j \le \ell(I)-1}$ implies that ${\ds \Ext_{R}^{\,j+1}(I^j, R)= 0}$ for every $j$ in the same range. Hence, from the isomorphisms above we deduce that ${\ds \Ext_{R}^{\,j+1}(E^j, R)= 0}$ for every ${\ds 1 \le j \le \ell(I)-1= \ell -e}$. Moreover, since $\r(I) \le \ell(I) -1$, $\R(I)$ is Cohen-Macaulay by  \cite[Theorem~3.1 and Corollary~3.4]{JU96}, see also \cite[Theorem~1.1]{U96}. Therefore, \cite[Theorem~3.5]{SUV03} guarantees that $\R(E)$ is Cohen-Macaulay and $\r(E)=\r(I)$. The equivalence of the three conditions in the statement now follows from Theorem~\ref{core}.
\end{proof}

A similar statement as in Theorem~\ref{balancedI} holds if we assume that $I$ is a strongly Cohen-Macaulay ideal of height two.

\begin{Corollary}\label{scmh2}
   Let $R$ be a Gorenstein local domain of dimension $d$ with infinite residue field. Let $I$ be a strongly Cohen-Macaulay ideal of height two, analytic spread $\ell(I) \ge 3$, and reduction number $\r(I) \le \ell(I) -1$. Assume that $I$ satisfies $\G_{\ell(I)}$.  Then, for $e>0$, the $R$-module $E = I \oplus R^{e-1}$ has reduction number $\r(E) = \r(I)$ and  $\rees{E}$ is Cohen-Macaulay. Moreover, the following are equivalent{\rm \,:}

\begin{enumerate}[{\rm (i)}]
\item $(U:_RE)E \subseteq \core(E)$ for every minimal reduction $U$ of $E${\rm \,;}
\item $(U:_RE)U=(U:_RE)E=\core(E)$ for every minimal reduction  $U$ of $E${\rm \,;}
\item$(U:_RE)$ is independent of $U$ for every minimal reduction $U$ of $E$.
\end{enumerate} 
\end{Corollary}

\begin{proof}
  Since $I$  is strongly Cohen-Macaulay of height two and satisfies $\G_{\ell(I)}$, then $\depth(I^j) \ge d-j$ for ${\ds 1 \le j \le \ell(I)-1}$ by \cite[Remark~2.10]{U94}. The conclusion now follows from Theorem~\ref{balancedI}.
\end{proof}

In the following corollary, we impose an additional condition on the ideal $I$ that results in the special fiber ring of $E$ being Cohen-Macaulay.

\begin{Corollary}\label{h2depthcase}
   Let $R$ be a Gorenstein local domain of dimension $d$ with infinite residue field. Let $I$ be an $R$-ideal of height two, analytic spread $\ell(I) \ge 3$, and reduction number $\r(I) \le \ell(I) -1$. Assume that $I$ satisfies $\G_{\ell(I)}$, that $\depth(I^j) \ge d-j$ for ${\ds 1 \le j \le \ell(I)-1}$, and that $I\m = J\m$ for some minimal reduction $J$ of $I$. Then, for $e>0$, the $R$-module $E = I \oplus R^{e-1}$ has reduction number $\r(E) = \r(I)$. Moreover, $\R(E)$ and $\F(E)$ are both Cohen-Macaulay, and the following are equivalent{\rm \,:}

\begin{enumerate}[{\rm (i)}]
\item $(U:_RE)E \subseteq \core(E)$ for every minimal reduction $U$ of $E${\rm \,;}
\item $(U:_RE)U=(U:_RE)E=\core(E)$ for every minimal reduction  $U$ of $E${\rm \,;}
\item$(U:_RE)$ is independent of $U$ for every minimal reduction $U$ of $E$.
\end{enumerate} 
\end{Corollary}

\begin{proof}
  From \cite[Theorem~4.8]{M15} it follows that $\R(I)$ and $\F(I)$ are both Cohen-Macaulay. Then, \cite[Theorem~3.5]{SUV03} implies that $\R(E)$ is Cohen-Macaulay, with $\r(I) = \r(E) \le \ell(I) -1 = \ell(E) - e$. This implies that the assumptions of Theorem~\ref{balancedI} are satisfied, which means the equivalence of the three conditions in the statement holds. Moreover, since $\R(E)$ and $\F(I)$ are both Cohen-Macaulay, then \cite[Theorem~4.8]{C1} implies that $\F(E)$ is Cohen-Macaulay.
\end{proof}

When $R$ is a Gorenstein local ring that is not a domain, we can still relate the equivalence of conditions (i)--(iii) in Theorem~\ref{core} with a requirement that the reduction number of $E$ is bounded, similarly as in Theorem~\ref{TheoremCPU02} for ideals. 

\begin{Corollary}\label{ModCorWithdepth}
   Let $R$ be a Gorenstein local ring of dimension $d$ with infinite residue field. Let $E$ be a finite, torsionfree and orientable $R$-module with $\rank(E)=e>0$, and analytic spread $\ell(E)=\ell \ge e+2$. Assume that $E$ is $\G_{\ell-e+1}$, that $\r(E) \le \ell -1$ and that ${\ds \depth(E^j) \ge d -j}$ for every ${\ds 1 \le j \le \ell-e}$. Then, the following are equivalent{\rm \,:}

\begin{enumerate}[{\rm (i)}]
\item $(U:_RE)E \subseteq \core(E)$ for every minimal reduction $U$ of $E${\rm \,;}
\item $(U:_RE)U=(U:_RE)E=\core(E)$ for every minimal reduction  $U$ of $E${\rm \,;}
\item$(U:_RE)$ is independent of $U$ for every minimal reduction $U$ of $E$.
\end{enumerate}
\end{Corollary}

\begin{proof}
  Since $R$ is Gorenstein, the assumption that ${\ds \depth(E^j) \ge d -j}$ for every ${\ds 1 \le j \le \ell-e}$ implies that ${\ds \Ext^{j+1}_R (E^j, R) =0}$ for all $j$ in the same range. Moreover, from \cite[Theorem~4.3]{C2} it follows that $\R(E)$ is Cohen-Macaulay. Then, all assumptions of Theorem~\ref{core} are satisfied, whence the remaining statements follow. 
\end{proof}

\medskip

In Theorem~\ref{TheoremCPU02}, the requirement that the reduction number is bounded is equivalent to the above characterization of the core. The proof ultimately relies on a thorough study of the canonical module of the associated graded ring of $I$ (see \cite[Corollary~2.4 and Corollary~2.6]{U96} and the proof of \cite[Theorem~2.6]{CPU02}). Since no module analogue exists for the notion of associated graded ring, a proof of a similar implication in the module case would require a completely different approach. 

\section{Acknowledgments}
The first named author would like to thank Mark Johnson for insightful conversations on residual intersections.


\begin{thebibliography}{99}

\bibitem{AN72}{M. Artin and M. Nagata, Residual intersections in Cohen-Macaulay rings, J. Math. Kyoto Univ. \textbf{12} (1972), 307-323.}

\bibitem{Avramov}{L. Avramov, Complete intersections and symmetric algebras, J. Algebra \textbf{73} (1981), 248-263.}

\bibitem{BH}{W. Bruns and J. Herzog, Cohen--Macaulay Rings, Cambridge Studies in Advanced Mathematics, \textbf{39}. Cambridge University Press, Cambridge, 1993.}

\bibitem{BAM20}{C. Bivi\`a-Ausina and J. Monta\~no, Analytic spread and integral closure of integrally decomposable modules, to appear in Nagoya Math. J., \texttt{arXiv:2001.08313}.}

\bibitem{BE77}{D. Buchsbaum and D. Eisenbud, What annihilates a module?, J. Algebra  \textbf{47} (1977), 231--243.}

\bibitem{CEU01}{M. Chardin, D. Eisenbud and B. Ulrich, Hilbert functions, residual intersections and residually $S_2$ ideals, Compos. Math. \textbf{125} (2001), 193-219.}

\bibitem{CPU01}{A. Corso, C. Polini and B. Ulrich, The structure of the core of ideals, Math. Ann. \textbf{321} (2001), 89--105.}

\bibitem{CPU02}{A. Corso, C. Polini and B. Ulrich, Core and residual intersections of ideals, Trans. Amer. Math. Soc. \textbf{354} (2002), 2579-2594.}

\bibitem{CPU03}{ A. Corso, C. Polini and B. Ulrich, Core of projective dimension one modules, Manuscripta Math. \textbf{111} (2003), 427-433.}

\bibitem{C1}{A. Costantini, Cohen-Macaulay fiber cones and defining ideal of Rees algebras of modules, To appear in Women in Commutative Algebra -- Proceedings of the 2019 WICA Workshop, Springer (2022), \textit{preprint} arXiv:2011.08453.}

\bibitem{C2}{A. Costantini, Residual intersections and modules with Cohen-Macaulay Rees algebras, J. Algebra \textbf{587} (2021), 36-63.}

\bibitem{EHU03}{D. Eisenbud, C. Huneke and B. Ulrich, What is the Rees algebra of a module? Proc. Amer. Math. Soc. \textbf{131} (2003), 701-708.}

\bibitem{FM12}{L. Fouli and S. Morey, Minimal reductions and cores of edge ideals, J. Algebra \textbf{364} (2012), 52-66.}
  
\bibitem{HSV83}{J. Herzog, A. Simis, and W. Vasconcelos, Approximation complexes of blowing-up rings, II, J. Algebra \textbf{82} (1983), 53-83.}

\bibitem{HVV85}{J. Herzog, W. Vasconcelos, and R. Villarreal, Ideals with sliding depth, Nagoya Math. J. \textbf{99} (1985), 159-172.}

\bibitem{Huneke82}{C. Huneke, Linkage and Koszul homology of ideals, Amer. J. Math. \textbf{104} (1982), 1043-1062.}

\bibitem{Huneke83}{C. Huneke, Strongly Cohen-Macaulay schemes and residual intersections, Trans. Amer. Math. Soc. \textbf{277} (1983), 739-763.}

\bibitem{HSwanson95}{C. Huneke and I. Swanson, Cores of ideals in 2-dimensional regular local rings, Michigan Math J. \textbf{42} (1995), 193-208.}

\bibitem{JU96}{M. Johnson and B. Ulrich, Artin-Nagata properties and Cohen-Macaulay associated graded rings, Compositio Math. \textbf{103} (1996), 7-29.}

\bibitem{Lipman94}{J. Lipman, Adjoints of ideals in regular local rings, Math. Research Letters \textbf{1} (1994), 739-755.}

\bibitem{M15}{J. Monta\~no, Artin-Nagata properties, minimal multiplicities, and depth of fiber cones, J. Algebra \textbf{425} (2015), 423-449.}


\bibitem{PX13}{C. Polini and Y. Xie, j-multiplicity and depth of associated graded modules, J. Algebra \textbf{379} (2013), 31-49.}

\bibitem{RS88}{D. Rees and J.D. Sally, General elements and joint reductions, Michigan Math. J. \textbf{35} (1988), 241-254.}

\bibitem{SUV03}{A. Simis, B. Ulrich and W. Vasconcelos, Rees algebras of modules, Proc. London Math. Soc. \textbf{87} (2003), 610-646.}

\bibitem{U94}{ B. Ulrich, Artin-Nagata properties and reductions of ideals, Contemp. Math. \textbf{159} (1994), 373-400.}

\bibitem{U96}{B. Ulrich, Ideals having the expected reduction number, Amer. J. Math \textbf{118} (1996), 17-38.}

\bibitem{Vil}{R. H. Villarreal, Rees algebras of edge ideals, Comm. Algebra \textbf{23} (1995), 3513--3524.}
\end{thebibliography}
\end{document}